\documentclass[12pt,twoside]{amsart}
\usepackage{amsfonts}

\def\rd{\mathbb{R}^d}
\def\Im{\mathop{\rm Im}\nolimits}
\def\Re{\mathop{\rm Re}\nolimits}

\def\R{\mathbb R}
\def\C{\mathbb C}
\def\N{\mathbb N}
\def\Im{\mathop{\rm Im}\nolimits}
\def\Re{\mathop{\rm Re}\nolimits}

\def\R{\mathbb R}
\def\C{\mathbb C}
\def\N{\mathbb N}
\def\ds{\displaystyle}
\def\cS{\mathcal{S}}

\newcommand{\beqsn}{\arraycolsep1.5pt\begin{eqnarray*}}
\newcommand{\eeqsn}{\end{eqnarray*}\arraycolsep5pt}
\newcommand{\beqs}{\arraycolsep1.5pt\begin{eqnarray}}
\newcommand{\eeqs}{\end{eqnarray}\arraycolsep5pt}

\newtheorem{Th}{Theorem}[section]
\newtheorem{Rem}[Th]{Remark}

\newtheorem{Lemma}[Th]{Lemma}

\newtheorem{Prop}[Th]{Proposition}

\def\pxi{\langle \xi \rangle}
\def\px{\langle x \rangle}
\def\pd{\langle D \rangle}
\def\der{\partial_{\xi}^{\alpha}\partial_{x}^{\beta}}
\catcode`\@=11

\renewcommand{\section}%
   {\setcounter{equation}{0}\@startsection {section}{1}{\z@}{-3.5ex plus -1ex
  minus -.2ex}{2.3ex plus .2ex}{\Large\bf}}

\title{Weighted energy estimates \\ for $p$-evolution equations}
\author[Ascanelli]{Alessia Ascanelli}
\address{Alessia Ascanelli\\
Dipartimento di Matematica ed Informatica\\Universit\`a di Ferrara\\
Via Machiavelli 35\\
44121 Ferrara\\
Italy}
\email{alessia.ascanelli@unife.it}
\author[Cappiello]{Marco Cappiello}
\address{Marco Cappiello\\
Dipartimento di Matematica ``G. Peano" \\Universit\`a di Torino\\
Via Carlo Alberto 10\\
10123 Torino\\
Italy}
\email{marco.cappiello@unito.it}

\begin{document}

\def\thefootnote{}

\begin{abstract}
We prove energy estimates for linear $p$-evolution equations in weighted Sobolev
spaces under suitable assumptions on the behavior at infinity of the coefficients
with respect to the space variables. As a consequence we obtain well posedness for
the related Cauchy problem in the Schwartz spaces $\mathcal{S}$ and $\mathcal{S}'$.
\end{abstract}

\maketitle

\markboth{\sc Well-posedness of the Cauchy problem for $p$-evolution 
equations}{\sc A.~Ascanelli, M. Cappiello}

\section{Introduction}
Let us start by considering the Cauchy problem 
\beqs
\label{CP}
\begin{cases}
P(t,x,D_t,D_x)u(t,x)=f(t,x) & (t,x)\in[0,T]\times\R,\cr
u(0,x)=g(x) & x\in\R,
\end{cases}
\eeqs
$D=-i \partial$, where $P(t,x,D_t,D_x)$ is a differential evolution operator of the form 
\begin{equation}\label{opdiff}
P(t,x,D_t,D_x)=D_t+a_p(t)D_x^p+\sum_{j=0}^{p-1}
a_j(t,x) D_x^j\,,
\end{equation}
with $p \in \N, p \geq 2,$ $a_p\in C([0,T];\R)$ and $a_j\in C([0,T];\mathcal B^\infty)$, where $\mathcal B^\infty$ stands for the class of complex valued $C^\infty(\R_x)$ functions with uniformly bounded derivatives.

Operators of the form above are usually referred to as ``$p$-evolution 
operators"; the condition that $a_p(t)$ is real valued means that the principal symbol of $P$
(in the sense of Petrowski) 
has the real characteristic $\tau=-a_p(t)\xi^p$; by the Lax-Mizohata 
theorem, this is a necessary condition to have a unique 
solution in Sobolev spaces of the Cauchy problem \eqref{CP} in a 
neighborhood of $t=0$, for any $p\geq1$. Notice that in the case $p =1$ operator \eqref{opdiff} is strictly hyperbolic; in the case $p =2$, operators of the form \eqref{opdiff} with real characteristics are usually called ``Schr\"odinger-type evolution operators", being the Schr\"odinger operator the most relevant model in the class. 

A wide literature concerning the well-posedness of problem \eqref{CP} 
in Sobolev spaces exists for $p=1,2.$ For general $p \geq 2,$ many results are known 
when the coefficients $a_j(t,x)$ are real valued, see for instance \cite{A1, A2, AZ, AC, CC2, CHR}.
When the coefficients $a_j(t,x)$ are complex valued for some $1\leq j\leq p-1$, then we know from \cite{I1} that some decay conditions for $|x|\to\infty$ must be required on the imaginary part of the coefficients in order to obtain $H^{\infty}$-well posedness. In the papers 
\cite{I1,I2}, Ichinose has given necessary and sufficient conditions for the case $p=2$, $x\in\R$. Kajitani and Baba \cite{KB} then proved that, for $p=2$ and $a_2(t)$ constant, $\ x\in\R^n$, the Cauchy problem \eqref{CP} is $H^{\infty}$-well posed if 
\begin{equation}\label{kbab}
\Im a_1(t,x)=\mathcal O(|x|^{-\sigma}),\; \sigma\geq 1,\; 
{\rm as}\; |x|\to\infty,
\end{equation} uniformly with respect to $t\in[0,T]$. Second order equations with $p=2$ and decay conditions as $|x|\to\infty$ have been considered, for example, in \cite{ACC, CR}.
Cicognani and Colombini \cite{CC} treated the case $p=3$ proving $H^{\infty}$-well posedness under the conditions
\begin{equation} \label{assCC}\begin{array}{l}
 |\Im a_2|\leq Ca_3(t)\langle x\rangle^{-1},
 \\
|\Im a_1|+|\Re \partial_x a_2|\leq Ca_3(t)\langle x\rangle^{-1/2}.
\end{array} \end{equation}
Recently, the first author et al. \cite{ABZ} extended the results of \cite{CC} and \cite{KB} to the case $p \geq 4$, 
giving sufficient conditions for $H^\infty$ well posedness of the Cauchy problem for the operator \eqref{opdiff};
results in \cite{ABZ} have then been generalized to pseudo-differential systems in \cite{AB}, and to higher order equations in \cite{AB2};  semi-linear 3-evolution equations have been then studied in \cite{ABZ2}.

\medskip

In this paper we want to consider the Cauchy problem \eqref{CP}
when $P(t,x,D_t,D_x)$ is an evolution operator of the form 
\begin{equation}\label{op}
P(t,x,D_t,D_x)=D_t+a_p(t,D_x)+\sum_{j=0}^{p-1}
a_j(t,x, D_x)\,,
\end{equation}
$p \in \N, p \geq 2,$ where $a_j$ are pseudo-differential operators with symbols $a_j$ of order $j$ for $0\leq j\leq p$, and for every $t\in[0,T]$, $x,\xi\in\R$ we have: $a_p(t,\xi)\in\R$,  $a_j(t,x,\xi)\in \C$ for $0\leq j\leq p-1$. 

In  \cite{AB} it has been proved the following:
\begin{Th}
The Cauchy problem (\ref{CP}) for the operator (\ref{op}) is $H^\infty$ well posed under the assumptions:
\beqs
\label{1}
|\partial_\xi a_p(t,\xi)|\geq C_p|\xi|^{p-1}\qquad\forall t\in[0,T],
\ |\xi|\gg1,
\eeqs
for some $C_p>0$ and
\beqs
\label{3}
&&|\Im\partial_\xi^\alpha a_j(t,x,\xi)|\leq C_\alpha\langle x
\rangle^{-\frac j{p-1}}\langle \xi\rangle_h^{j-\alpha},\ 1\leq j\leq p-1\\
\label{4}
&&|\Im \partial_\xi^\alpha D_xa_j(t,x,\xi)|
\leq C_\alpha\langle x
\rangle^{-\frac{j-1}{p-1}}\langle\xi\rangle_h^{j-\alpha},\ 2\leq j\leq p-1\\
\label{5}
&&|\Im \partial_\xi^\alpha D_x^\beta a_j(t,x,\xi)|\leq C_\alpha
\langle x\rangle^{-\frac{j-[\beta/2]}{p-1}}\langle\xi\rangle_h^{j-\alpha},
\ 1\leq\left[\frac\beta2\right]\leq j-1,\ 3\leq j\leq p-1
\eeqs
for all $(t,x,\xi)\in[0,T]\times\R^2$ and for some $C_\alpha>0$, where $\langle \cdot\rangle_h=\sqrt{h^2+|\cdot|^2}$, $h\geq 1$.  More precisely, there exists $\sigma>0$ such that for all
$f\in C([0,T];H^s)$ and $g\in H^s$ there is a unique solution
$u\in C([0,T];H^{s-\sigma})$ of \eqref{CP}, \eqref{op} which satisfies the following energy
estimate:
\beqs
\label{Eab}
\|u(t,\cdot)\|^2_{s-\sigma}\leq C_s\left(\|g\|^2_s+
\int_0^t\|f(\tau,\cdot)\|^2_s \,d\tau\right)\qquad
\forall t\in[0,T],
\eeqs
for some $C_s>0$. 
\end{Th}
Formula \eqref{Eab} shows that the Cauchy problem \eqref{CP}, \eqref{op} is $H^\infty$ well posed with loss of $\sigma$ derivatives, in the sense that the solution is less regular then the Cauchy data. This phenomenon, which is usual in the theory of degenerate hyperbolic equations, appears so also in the theory of non-degenerate p-evolution equations for $p \geq 2$, and has been yet observed in \cite{CC2, I2, KB}. Notice that assumptions \eqref{1}-\eqref{5} are consistent with the conditions in \eqref{kbab}, \eqref{assCC}. 
The loss of derivatives appearing in \eqref{Eab} for the solution of (\ref{CP}) is explicitly computed in \cite{AB}, and it can be avoided by slightly strenghtening the sole assumption (\ref{3}) for $j=p-1$.
Formula (\ref{Eab}) gives so an accurate information about the regularity of the solution, but, in spite of the very precise decay conditions on the coefficients, it does not say anything about the behavior of the solution as $|x|\to\infty$.

This suggests us to change the setting of the Cauchy problem (\ref{CP}) to gain the possibility of
giving similar precise information on the behavior of the solution for $|x| \to \infty;$ namely, one
could try to obtain energy estimates in suitable weighted Sobolev spaces and well-posedness in
the Schwartz spaces $\mathcal{S}(\R), \mathcal{S}'(\R).$ 
\\
Results of the 
above type have been proved for strictly hyperbolic equations $(p=1)$ by Cordes \cite{Co}; we also 
recall similar results when the coefficients are not Lipschitz continuous in $t$, see \cite{AsCa, AsCa2}.
The natural framework consists in dealing with pseudo differential operators with symbols in the classes $SG^{m_1,m_2}=SG^{m_1,m_2}(\R^2)$,
with $m_j \in \R, j=1,2$, defined as the  
class of all functions $p(x,\xi) \in C^{\infty}(\R^{2})$ satisfying the following estimates:
\begin{equation} \label{symbols}
\| p\|_{\alpha, \beta }:= \sup_{(x,\xi) \in \R^{2}}\pxi^{-m_{1}+\alpha}
\px^{-m_{2}+\beta}|\der p(x,\xi)| <\infty 
\end{equation}
for every $\alpha, \beta \in \N.$ This class can be considered as a particular case of general H\"ormander classes, see
\cite[Chapter XVIII]{Ho}. A detailed specific calculus for this class is presented in \cite{Co}, \cite{Pa}. 
In the following we shall prove energy estimates in the weighted Sobolev spaces $H_{s_1,s_2}(\R)$, $s_j \in \R,j=1,2,$ defined as the space of all $u \in \mathcal{S}'(\R)$ satisfying the following condition:
\begin{equation} \label{sobolev} \|u \|_{s_1,s_2}=\|\px^{s_{2}}\pd^{s_{1 }}u\|_{L^{2}} <\infty,
\end{equation}
where we denote by $\pd^{s_1}$ the Fourier multiplier with symbol $\pxi^{s_1}$.
It is worth to recall that for $s_{2}=0$ we recapture the standard Sobolev spaces and that the following identities hold:
\begin{equation}\label{S}
\bigcap_{s_1,s_2 \in \R}H_{s_1,s_2}(\R)= \mathcal{S}(\R), \qquad \qquad \bigcup_{s_1,s_2 \in \R }H_{s_1,s_2}(\R)= \mathcal{S}'(\R).
\end{equation}
Moreover we recall that $\mathcal{S}(\R)$ is dense in $H_{s_1,s_2}(\R)$ for any $s_1,s_2 \in \R.$

\medskip
The main result of the paper is the following.

\begin{Th}\label{main}
Let $P(t,x,D_t,D_{x})$ be an operator of the form \eqref{op} and assume that the following conditions hold:
\beqs
\label{assAC1}
&& a_p\in C([0,T]; SG^{p,0}),
\\
&&
\label{assAC2}
|\partial_\xi a_p(t,\xi)|\geq C_p|\xi|^{p-1}\qquad\forall t\in[0,T],
\ |\xi|\gg1,\ {\it with\ } C_p>0,
\\
\label{assAC3}
&&a_{j} \in C([0,T]; SG^{j,-j/(p-1)}), \quad j=0,\ldots,p-1. 
\eeqs
Then, the Cauchy problem (\ref{CP})
is well-posed in $\mathcal{S}(\R), \mathcal{S}'(\R)$. More precisely,
there exists $\sigma>0$ such that for all $s_{1},s_{2} \in \R$,
$f\in C([0,T];H_{s_1,s_2}(\R))$ and $g\in H_{s_1,s_2}(\R)$ there is a unique solution
$u\in C([0,T];H_{s_1,s_2-\sigma}(\R))$ which satisfies the following energy
estimate:
\beqs
\label{E}
\|u(t,\cdot)\|^2_{s_1,s_2-\sigma}\leq C_s\left(\|g\|^2_{s_1,s_2}+
\int_0^t\|f(\tau,\cdot)\|^2_{s_1,s_2} \,d\tau\right)\qquad
\forall t\in[0,T],
\eeqs
for some $C=C(s_1,s_2)>0$.
\end{Th}
\begin{Rem}
{\rm
The energy estimate \eqref{E} shows that we can obtain well-posedness in $\mathcal{S}(\R)$ and$ \mathcal{S}'(\R)$ {\it without any loss of derivatives} for the solution $u$ of \eqref{CP}, paying this with a modification of the rate of decay/growth at infinity of the solution with respect to the Cauchy data. The solution of \eqref{CP} has so the same regularity as the Cauchy data; but, if we start from data with a prescribed decay at infinity, then a loss of decay appears in the solution; similarly, if the data have a fixed polynomial growth at infinity, then the solution presents a stronger growth. 
}
\end{Rem}
\begin{Rem}
{\rm
The proof of Theorem \ref{main} is in part inspired by \cite{ABZ}, but it takes advantage of the fact that in the new framework we are considering we can admit initial data with polynomial growth with respect to the space variable. We also observe that if we approach the problem by using exactly the technique of \cite{ABZ} we obtain well posedness in $\mathcal{S}(\R), \mathcal{S}'(\R)$ with the same loss of derivatives as in \cite{ABZ}, whereas the behavior at infinity of the solution remains the same as the initial data. Thus, on one hand, if we can admit a solution of \eqref{CP} less regular then the data, we have it; on the other hand, if we are looking for a solution with the same regularity as the Cauchy data, we can find it but this solution presents a ``worse'' behavior at infinity. Moreover the solution exists uniquely and the precise value of $\sigma$ is computed, see formulas \eqref{estloss} and \eqref{choice!}.
}
\end{Rem}

\begin{Rem}
{\rm If the condition (\ref{assAC3}) with $j=p-1$ is strengthened into $$a_{p-1}\in C([0,T]; SG^{p-1, -(1+\epsilon)}),$$ for any $\epsilon>0$, then the Cauchy problem (\ref{CP}) is well posed in $\mathcal S(\R)$, $\mathcal  S'(\R)$ without loss of derivatives and without modification of the behavior at infinity. }
\end{Rem}

\begin{Rem}{\rm 
We observe that the assumption \eqref{assAC3} in Theorem \ref{main} can be slightly weakened without changing the argument of the proof. Namely we can replace the condition \eqref{assAC3} with the following 
\beqs\label{refinement}
\Re a_j\in C([0,T]; SG^{j,0}), \;\Im a_j\in C([0,T]; SG^{j,-j/(p-1)}), \qquad 0\leq j\leq p-1.
\eeqs
The argument of the proof remains essentially the same but it involves more complicate notation. For this reason we prefer to present our main result using the more simple assumption \eqref{assAC3}. We refer to Remark \ref{finalrem} at the end of the paper for some comments on the more refined result. \\
%
Finally we observe that if $a_{j}(t,x,D_{x})$ are differential operators, assumptions \eqref{assAC1}, \eqref{assAC2}, \eqref{refinement} are consistent with the ones given in \cite{ABZ, CC, KB} for the corresponding case $a_p(t)>C_p>0$ $\forall t\in[0,T]$.}
\end{Rem}


\begin{section}{Preliminaries}
In this section we collect some basic notions on SG classes of pseudo-differential operators and prove some preliminary results which will be used in the proof of Theorem \ref{main} in the next section.
\subsection{SG-pseudo-differential operators}
A detailed exposition on the classes $SG^{m_{1},m_{2}}$ defined in the Introduction can be found in \cite{Co}. Here we recall only some 
basic facts which will be used in the proof of our result. In general, fixed $d \in \N \setminus \{0\}$, the space $SG^{m_{1},m_{2}}(\R^{2d})$ is 
the space of all functions $p(x,\xi) \in C^{\infty}(\R^{2d})$ satisfying the following estimates:
\begin{equation} \label{dsymbols}
\| p\|_{\alpha, \beta }:= \sup_{(x,\xi) \in \R^{2d}}\pxi^{-m_{1}+|\alpha|}
\px^{-m_{2}+|\beta|}|\der p(x,\xi)| <\infty 
\end{equation}
for every $\alpha, \beta \in \N.$ We can associate to every $p \in SG^{m_1,m_2}(\R^{2d})$ a pseudo-differential operator defined by 
\begin{equation} \label{pseudooperator}
Pu(x)=p(x,D)u(x) = (2\pi)^{-d} \int_{\rd} e^{i\langle x,\xi\rangle} p(x,\xi) \hat{u}(\xi)\, d\xi.
\end{equation}
The operator $p(x,D)$ is a linear continuous map  $\mathcal{S}(\rd)\to\cS(\rd)$ which
extends to a continuous map $\mathcal{S}'(\rd)\to\cS'(\rd)$.  
Concerning the action of these operators on weighted Sobolev spaces, we have that if $p\in SG^{m_{1},m_{2}}(\R^{2d})$, then the map $$p(x,D): H_{s_{1},s_{2}}(\rd) \to H_{s_{1}-m_{1},s_{2}-m_{2}}(\rd)$$ is continuous for every $s_{1}, s_{2} \in \R,$ where the space $H_{s_1,s_2}(\rd)$ is obviously defined in arbitrary dimension as the space of all $u \in \mathcal{S}'(\rd)$ satisfying \eqref{sobolev}. We also recall the following result concerning the composition and the adjoint of SG operators.
\begin{Prop}
\label{composition}
Let $p \in SG^{m_{1},m_{2}}(\R^{2d})$ and $q \in SG^{m'_{1},m'_{2}}(\R^{2d})$. Then there exists a symbol $s \in SG^{m_{1}+m'_{1},m_{2}+m'_{2}}(\R^{2d})$
such that $p(x,D)q(x,D) = s(x,D)+R$ where $R$ is a smoothing operator $\cS'(\rd) \to \cS(\rd).$ Moreover, $s$ has the following asymptotic expansion 
$$s(x,\xi) \sim \sum_{\alpha} \alpha!^{-1}\partial_{\xi}^{\alpha}p(x,\xi) D_{x}^{\alpha}q(x,\xi)$$ i.e. for every $N \geq 1,$ we have
$$s(x,\xi) - \sum_{|\alpha|<N} \alpha!^{-1}\partial_{\xi}^{\alpha}p(x,\xi) D_{x}^{\alpha}q(x,\xi) \in SG^{m_{1}+m'_{1}-N,m_{2}+m'_{2}-N}(\R^{2d}).$$
\end{Prop}   

\begin{Prop} Let $p \in SG^{m_{1},m_{2}}(\R^{2d})$ and let $P^{\ast}$ be the $L^{2}$-adjoint of $p(x,D)$. Then there exists a symbol $p^{\ast} \in SG^{m_{1},m_{2}}(\R^{2d})$ such that $P^{\ast}=p^{\ast}(x,D)+R'$, where $R'$ is a smoothing operator $\cS'(\rd) \to \cS(\rd).$ Moreover, $p^{\ast}$ has the following asymptotic expansion 
$$p^{\ast}(x,\xi) \sim \sum_{\alpha} \alpha!^{-1}\partial_{\xi}^{\alpha}D_{x}^{\alpha}\overline{p(x,\xi)}$$ i.e. for every $N \geq 1,$ we have
$$p^{\ast}(x,\xi) - \sum_{|\alpha|<N} \alpha!^{-1}\partial_{\xi}^{\alpha}D_{x}^{\alpha}\overline{p(x,\xi)} \in SG^{m_{1}-N,m_{2}-N}(\R^{2d}).$$
\end{Prop}
We also recall the definition of the class $S^m(\R^{2d}), m \in \R,$ defined as the space of all symbols $p(x,\xi) \in C^{\infty}(\R^{2d})$ satisfying
$$|\der p(x,\xi)| \leq C_{\alpha \beta}\pxi^{m-|\alpha|}, \qquad (x,\xi) \in \R^{2d}$$ for every $\alpha, \beta \in \N^d.$
It is important for the sequel to notice that 
\begin{equation}\label{inclusion}
SG^{m_1,m_2}(\R^{2d}) \subset S^{m_1}(\R^{2d}) \end{equation}
for any $m_1,m_2 \in \R$ with $m_2 \leq 0$ and that the operators with symbols in $S^0(\R^{2d})$ map continuously $H_{s_1,s_2}(\rd)$ to itself for every $s_1,s_2 \in \R.$ 
\par 
In the proof of Theorem \ref{main} we shall also use the sharp G{\aa}rding inequality applied to SG operators. This result is known as a particular case of \cite[Thm. 18.6.14]{Ho}. However, for our purposes we also need a precise estimate of the order of the remainder with respect to $\xi$, which has been proved only for symbols in the H\"ormander classes $S^{m}(\R^{2d})$, see \cite[Thm. 4.2]{KG}. Nevertheless, we have to observe that the operators we shall consider have negative order with respect to $x$. Hence, in view of the inclusion \eqref{inclusion}, we can base the proof of this result on the classical sharp G{\aa}rding inequality for standard H\"ormander symbols and estimate the order of the remainder term with respect to $\xi$ by looking at its classical asymptotic expansion. Namely we have the following result.

\begin{Th}\label{tha1}
Let $m_1\geq 0, m_2\leq 0,$ $a \in SG^{m_1,m_2}(\R^{2d})$ with $\Re a(x,\xi)\geq0$.
Then there exist pseudo-differential operators $Q=q(x,D)$, $\tilde{R}=\tilde{r}(x,D)$ and $R_0=r_0(x,D)$
with symbols, respectively, $q \in SG^{m_1,m_2}(\R^{2d})$, $\tilde{r} \in SG^{m_1-1,m_2}(\R^{2d})$ and $r_0 \in S^0(\R^{2d})$
such that
\begin{equation}
a(x,D)=q(x,D)+\tilde{r}(x,D)+r_0(x,D) \label{2.29} 
\end{equation}
\begin{equation} \label{low}
\Re\langle q(x,D) u,u\rangle\geq0\qquad\forall u\in \cS(\rd).
\end{equation}
\end{Th}
\begin{proof} Since $m_2 \leq 0,$ then $SG^{m_1,m_2}(\R^{2d}) \subset S^{m_1}(\R^{2d})$. Hence, the classical G{\aa}rding inequality gives the existence of two symbols $q$ and $r$ such that $a(x,D) =q(x,D)+r(x,D)$ and $q(x,D)$ satisfies \eqref{low}. Let us now consider the asymptotic expansion of the remainder term $r(x,D).$ By Theorem 4.2 in \cite{KG}, we have that 
\begin{equation}\label{restosharpgarding}
r(x,\xi)\sim\psi_1(\xi)D_x a(x,\xi)+
\sum_{|\alpha+\beta|\geq2}\psi_{\alpha,\beta}(\xi) \partial_\xi^\alpha D_x^\beta a(x,\xi),
\end{equation}
for some real valued functions $\psi_1, \psi_{\alpha , \beta}$ with $\psi_1 \in SG^{-1,0}(\R^{2d})$ and $\psi_{\alpha, \beta} \in SG^{(|\alpha|-|\beta|)/2,0}(\R^{2d}).$
In particular, we have that
$$r(x,\xi)= \psi_1(\xi)D_x a(x,\xi)+
\sum_{2\leq |\alpha+\beta|\leq 2m_1-1}\psi_{\alpha,\beta}(\xi) \partial_\xi^\alpha D_x^\beta a(x,\xi) + r_0(x,\xi),$$
for a symbol $r_0 \in S^0(\R^{2d}).$
Moreover, it is easy to verify that $\psi(\xi)D_xa(x,\xi) \in SG^{m_1-1,m_2-1}(\R^{2d})$ and that 
$$\sum_{2\leq |\alpha+\beta| \leq 2m_1-1}\psi_{\alpha,\beta}(\xi) \partial_\xi^\alpha D_x^\beta a(x,\xi) \in SG^{m_1-1,m_2}(\R^{2d}).$$
Then we have that the symbol $$\tilde{r}(x,\xi) = \psi_1(\xi)D_x a(x,\xi)+
\sum_{2\leq |\alpha+\beta|\leq 2m_1-1}\psi_{\alpha,\beta}(\xi) \partial_\xi^\alpha D_x^\beta a(x,\xi) \in SG^{m_1-1,m_2}(\R^{2d}).$$ This concludes the proof.
\end{proof}

\begin{Rem} {\rm
In the sequel of the paper we will often replace the weight function $\pxi$ with $\pxi_h= (h^2+|x|^2)^{1/2}$ for some $h \geq 1$ to prove our results. It is clear that this modification does not change the definition of the class $SG^{m_1,m_2}(\R^{2d})$ and of the spaces $H^{s_1,s_2}(\rd),$ and their properties.}
\end{Rem}

\subsection{Changes of variables and conjugations}
The idea of the proof of Theorem \ref{main} is to prove an energy estimate in $L^2(\R)$ for the operator
\beqs\label{A}
iP=\partial_t+ia_p(t,D_x)+\displaystyle\sum_{j=0}^{p-1}ia_j(t,x,D_x)=\partial_t+A(t,x,D_x).
\eeqs
We have
\beqs
\nonumber
\frac{d}{dt}\|u\|_0^2=&&2\Re\langle\partial_t u,u\rangle
=2\Re\langle iPu,u\rangle-2\Re\langle Au,u\rangle\\
\label{56}
\leq&&\|f\|_0^2+\|u\|_0^2-2\Re\langle Au,u\rangle.
\eeqs
Notice that $2\Re\langle Au,u\rangle=\langle (A+A^\ast)u,u\rangle$, with $A^\ast$ the formal adjoint of $A$, and $A+A^\ast$ is an operator with symbol in $SG^{p-1,-1}$, hence with positive order with respect to $\xi$. This implies that the desired energy estimate is not straightforward and in order to obtain it, we need to transform the Cauchy problem \eqref{CP} into an equivalent one of the form
\beqs
\label{CP2}
\begin{cases}
P_\lambda u_{\lambda} =f_\lambda\\
u_{\lambda}(0,x)=g_\lambda,
\end{cases}
\eeqs
where $P_\lambda=D_t-iA_\lambda$ and $\Re A_\lambda(t,x,\xi)\geq 0$; then we apply Theorem \ref{tha1} to obtain the estimate from below $$\Re\langle A_\lambda v,v\rangle\geq -c||v||_0^2$$ for $v \in \mathcal{S}(\R)$ and for some positive constant $c$. This, computing as in \eqref{56}, will give an $L^2$ energy estimate for the solution $u_\lambda$ of the Cauchy problem \eqref{CP2}. The operator $P_{\lambda}$ will be the result of $p-1$ conjugations of $P$ with operators of the form $e^{\lambda_{p-k}(x,D_x)}, k=1,\ldots, p-1$, namely:
\begin{equation}
\label{Plambda}
(iP)_\lambda:=(e^{\lambda_{1}(x,D_x)})^{-1}\cdots(e^{\lambda_{p-2}(x,D_x)})^{-1}(e^{\lambda_{p-1}(x,D_x)})^{-1}(iP)e^{\lambda_{p-1}(x,D_x)}e^{\lambda_{p-2}(x,D_x)}\ldots e^{\lambda_{1}(x,D_x)}.
\end{equation}
Here and in the following we shall denote by $e^{\pm \lambda_{p-k}(x,D_x)}, k=1,\ldots, p-1$, the operators with symbols $e^{\pm \lambda_{p-k}(x,\xi)}$ and the functions $\lambda_{p-k}$ will be chosen such that:
\begin{itemize}
\item $\lambda_{p-k}(x,\xi)$ are real valued, $1\leq k\leq p-1$;
\item $e^{\lambda_{p-1}(x,\xi)}\in SG^{0, M_{p-1}}$ for some $M_{p-1}>0$ and $e^{\lambda_{p-k}(x,\xi)}\in SG^{0,0}$ for $2\leq k\leq p-1;$
\item the operator $e^{\lambda_{p-k}(x,D_x)}$ is invertible for every $1\leq k\leq p-1$ and the principal part of $(e^{\lambda_{p-k}(x,D_x)})^{-1}$ is $e^{-\lambda_{p-k}(x,D_x)}$;
\item the operator $$A_\lambda:= (e^{\lambda_{1}(x,D_x)})^{-1}\cdots(e^{\lambda_{p-2}(x,D_x)})^{-1}(e^{\lambda_{p-1}(x,D_x)})^{-1}(iA)e^{\lambda_{p-1}(x,D_x)}e^{\lambda_{p-2}(x,D_x)}\ldots e^{\lambda_{1}(x,D_x)}$$
is such that $\Re\langle A_\lambda v,v\rangle\geq-c||v||_0^2$ $\forall v(t,\cdot)\in\mathcal S(\R)$. \end{itemize}
After the transformation of the problem \eqref{CP} into \eqref{CP2} with $P_{\lambda}$ defined by \eqref{Plambda} and $f_{\lambda}$ and $g_{\lambda}$ given by
\begin{equation}
\label{fgdef}
f_{\lambda}=(e^{\lambda_{1}(x,D_x)})^{-1}\cdots(e^{\lambda_{p-1}(x,D_x)})^{-1}f, \quad g_{\lambda}=(e^{\lambda_{1}(x,D_x)})^{-1}\cdots(e^{\lambda_{p-1}(x,D_x)})^{-1}g,
\end{equation}
 we will obtain an energy estimate in $L^2(\R)$ for the new variable
\begin{equation}\label{vdef}
u_{\lambda}(t,x)= (e^{\lambda_{1}(x,D_x)})^{-1}\cdots(e^{\lambda_{p-2}(x,D_x)})^{-1}(e^{\lambda_{p-1}(x,D_x)})^{-1}u(t,x)
\end{equation}
which will yield to an estimate of the form \eqref{E} for the solution $u$ of \eqref{CP}. 
\bigskip 

Let us now define the functions $\lambda_j$. We set
\beqs
\label{26'}
\lambda_{p-1}(x,\xi):=M_{p-1}\omega\left(\frac \xi h\right)\ds\int_0^x\frac1{\langle y\rangle}dy,
\eeqs
and for $2\leq k\leq p-1$
\beqs
\label{26}
\lambda_{p-k}(x,\xi):=M_{p-k}\omega\left(\frac \xi h\right) \langle\xi\rangle_h^{-k+1} \int_0^x
\langle y\rangle^{-\frac{p-k}{p-1}}\psi\left(
\frac{\langle y\rangle}{\langle\xi\rangle_h^{p-1}}\right)dy
\,,
\eeqs
where $M_{p-1}, M_{p-2},\ldots, M_1$ are positive constants to be chosen later on, $\omega \in C^\infty(\R)$ is such that
\beqs\label{omegaaaaa}
&&\omega(\xi)=
\begin{cases}
0 & |\xi|\leq 1\\
{\rm sgn}(\partial_\xi a_p(t,\xi))& |\xi|\geq R
\end{cases}
\eeqs
for some $R>1$, and $\psi\in C^\infty_0(\R)$ is such that $0\leq\psi(y)\leq 1$ $\forall y\in\R$, $\psi(y)=1$ for $|y|\leq\frac 12$, $\psi(y)=0$ for $ |y|\geq1.$ Notice that assumption \eqref{assAC2} ensures the existence of $R>0$ such that for every fixed $\xi$ with $|\xi|>R$ the sign of the function $\partial_\xi a_p(t,\xi)$ remains constant for every $t\in [0,T]$, then $\omega$ is well defined and does not depend on $t$.

Definition \eqref{26} is in part inspired by \cite{AB, ABZ}; more precisely, the symbols $\lambda_{p-k}$ in \eqref{26} are exactly the same as in \cite{AB}, while the symbol $\lambda_{p-1}$ in \eqref{26'} is new: it can be considered only in the framework of the $SG$ calculus, where symbols with polynomial growth in $x$ can be handled. The setting we are using allows to construct a transformation with a "stronger" $\lambda_{p-1}$ with respect to \cite{AB} still remaining in (weighted) Sobolev spaces.

\begin{Lemma}
\label{lemma1}
The function $\lambda_{p-1}$ defined by \eqref{26'} satisfies the following estimates:
\beqs\label{7}
&&|\lambda_{p-1}(x,\xi)|\leq M_{p-1}(1+\ln\px)
\\
\label{8}
&&|\partial_\xi^\alpha\partial_x^\beta \lambda_{p-1}(x,\xi)|\leq M_{p-1} C_{\alpha,\beta}\langle x\rangle^{-\beta}
\langle\xi\rangle_h^{-\alpha}\quad \alpha\geq 0,\beta\geq 1,
\\
\label{800}
&&|\partial_\xi^\alpha \lambda_{p-1}(x,\xi)|\leq M_{p-1} C'_{\alpha,R}\pxi_h^{-\alpha}\left(1+\ln\px\chi_{E_{h,R}}(\xi)\right),\quad \alpha\geq 1,
\eeqs
with positive constants $C,C_{\alpha,\beta},C'_{\alpha,R}$, and where $\chi_{E_{h,R}}$ is the characteristic function of the set $E_{h,R}=\{\xi\in\R\vert\ h\leq|\xi|\leq hR\}$.
\end{Lemma}

\begin{proof} 
A simple explicit computation of the integral in \eqref{26'} gives
\beqs\label{this}
|\lambda_{p-1}(x,\xi)|\leq M_{p-1}\log(2\px)< M_{p-1}(1+\ln\px).
\eeqs
For $\beta\geq 1$ we have
\beqs
|\partial_x^\beta\lambda_{p-1}(x,\xi)|=M_{p-1}\left\vert\omega\left(\frac \xi h\right)\partial_x^{\beta-1}\px^{-1}\right\vert
\label{two}
\leq M_{p-1}C_\beta\px^{-\beta},
\eeqs 
and  for $\alpha,\beta\geq 1$: 
\beqs\nonumber
|\partial_\xi^\alpha\partial_x^\beta\lambda_{p-1}(x,\xi)|=M_{p-1}\left\vert\omega^{(\alpha)}\left(\frac \xi h\right)h^{-\alpha}\partial_x^{\beta-1}\px^{-1}\right\vert
&\leq &M_{p-1}C_{\alpha,\beta} h^{-\alpha}\px^{-\beta}
\\
\label{dermist}
&\leq &M_{p-1}C_{\alpha,\beta}\pxi_h^{-\alpha}\px^{-\beta}.
\eeqs
Finally, for $\alpha\geq 1$:
\beqs\nonumber
|\partial_\xi^\alpha\lambda_{p-1}(x,\xi)|&=&M_{p-1}\left\vert\omega^{(\alpha)}\left(\frac \xi h\right)h^{-\alpha}\int_0^x\frac1{\langle y\rangle} dy\right\vert
\\
\nonumber
&\leq& M_{p-1}C_\alpha h^{-\alpha}\ln\px\chi_{E_{h,R}}(\xi)
\\
\label{one}
&\leq& M_{p-1}C_{\alpha,R} \pxi_h^{-\alpha}\ln\px\chi_{E_{h,R}}(\xi),
\eeqs
since $h^{-1}\leq\langle R\rangle\pxi_h^{-1}$ on $E_{h,R}$.

\end{proof}

\begin{Lemma}\label{lemma1bis}
Let $\lambda_{p-k}, k=2,\ldots, p-1$ be defined by \eqref{26}. Then for every $\alpha, \beta\in\N$, there exists a constant $C_{k,\alpha,\beta}>0$ such that
\beqs\label{40}
|\partial_\xi^\alpha\partial_x^\beta\lambda_{p-k}(x,\xi)|&\leq&
C_{k,\alpha,\beta}M_{p-k}\langle x\rangle^{\frac{k-1}{p-1}-\beta}
\langle\xi\rangle_h^{-\alpha-k+1}\chi_{\xi}(x)
\\
\nonumber
&\leq& C_{k,\alpha,\beta}M_{p-k}\langle x\rangle^{-\beta}
\langle\xi\rangle_h^{-\alpha},
\eeqs
where $\chi_{\xi}(x)$ denotes the characteristic function of the set $\left\{x\in\R\ \vert\ \langle x\rangle\leq \langle\xi\rangle_h^{p-1}\right\}$.
In particular, we have that $\lambda_{p-k} \in SG^{0,0}$ for $2 \leq k\leq p-1.$
\end{Lemma}

\begin{proof} See \cite[Lemma 2.1]{ABZ}.
\end{proof}

From the estimates proved in Lemma \ref{lemma1} and \ref{lemma1bis} we obtain, by simply applying the Fa\`a di Bruno formula, the following estimates for the symbols $e^{\pm\lambda_{p-k}(x,\xi)}$. We leave the details of the proof to the reader.

\begin{Lemma}
\label{Lemma2.5}
Let $\lambda_{p-k}, k=1,\ldots, p-1$ be defined by \eqref{26'} and \eqref{26}. Then 
\beqs
\label{lambdasenzader}
|e^{\pm \lambda_{p-1}(x,\xi)}|&\leq& K\px^{M_{p-1}},
\\
\label{derlambdaxi}
|\partial_{\xi}^{\alpha}e^{\pm \lambda_{p-1}(x,\xi)}| &\leq& C_{\alpha}[1+\ln\px \chi_{E_{h,R}}(\xi)]^{\alpha}  \pxi^{-\alpha}_he^{\pm \lambda_{p-1}(x,\xi)},\ \alpha\geq 1,
\\
\label{derlambdax}
|\partial_{x}^{\beta}e^{\pm \lambda_{p-1}(x,\xi)}| &\leq &C_{\beta}\px^{-\beta}e^{\pm \lambda_{p-1}(x,\xi)},\ \beta\geq 1,
\\
\label{dermistelambdaxi}
|\partial_{\xi}^{\alpha}\partial_x^\beta e^{\pm \lambda_{p-1}(x,\xi)}| &\leq& C_{\alpha,\beta}
\left[1+\ln\px\chi_{E_{h,R}}(\xi)\right]^\alpha\px^{-\beta}\pxi^{-\alpha}_h e^{\pm \lambda_{p-1}(x,\xi)},
\qquad \alpha,\beta\geq 1,\\ 
\label{derlambdap-k}
|\partial_{\xi}^{\alpha}\partial_x^\beta e^{\pm \lambda_{p-k}(x,\xi)}| &\leq& C_{\alpha,\beta,k}
\px^{-\beta}\pxi^{-\alpha}_h e^{\pm \lambda_{p-k}(x,\xi)}, \qquad 2\leq k\leq p-1, \alpha,\beta \in \N,
\eeqs
for some positive constants $K, C_\alpha, C_\beta, C_{\alpha,\beta}, C_{\alpha,\beta,k}.$
\end{Lemma}

The next two results state the invertibility of the operators $e^{\lambda_{p-k}(x,D_x)}$ for $k=1,\ldots, p-1.$ 

\begin{Lemma}
\label{lemma2}
Let $\lambda_{p-1}(x,\xi)$ be defined by \eqref{26'}. Then there exists $h_1\geq 1$ such that for $h\geq h_1$
the operator $e^{\lambda_{p-1}(x,D)}$ is invertible and
\beqs
\label{18}
(e^{\lambda_{p-1}(x,D_x)})^{-1}=e^{-\lambda_{p-1}(x,D_x)}(I+R_{p-1})
\eeqs
where $I$ is the identity operator and $R_{p-1}$ has
principal symbol given by
\beqsn
r_{p-1,-1}(x,\xi)\in SG^{-1,0}.
\eeqsn
\end{Lemma}

\begin{proof}
By Proposition \ref{composition} it follows that $$
e^{\lambda_{p-1}(x,D_x)}e^{-\lambda_{p-1}(x,D_x)}=
I-r_{p-1,-1}(x,D) + r_{p-1,-2}(x,D),$$ where 
$$r_{p-1,-1}(x,\xi)= \partial_\xi \lambda_{p-1}(x,\xi)D_x\lambda_{p-1}(x,\xi)$$ and $$r_{p-1,-2}(x,\xi) \sim \sum_{m\geq2}\frac{1}{m!}\partial_\xi^me^{\lambda_{p-1}(x,\xi)}
D_x^m e^{-\lambda_{p-1}(x,\xi)}.$$
By \eqref{8} and \eqref{800}
\beqsn
r_{p-1,-1}(x,\xi)\in SG^{-1,-1+\varepsilon} \quad \textit{for every} \quad \varepsilon >0
\eeqsn
and
\beqsn
r_{p-1,-2}(x,\xi) \in SG^{-2,-2+\varepsilon} \quad \textit{for every} \quad \varepsilon >0.
\eeqsn
More precisely,
\beqs\nonumber
|\partial_\xi^\alpha D_x^\beta r_{p-1,-1}(x,\xi)|&&\leq \sum_{\alpha_1+\alpha_2=\alpha}\left(c_{\alpha_1,\alpha_2,\beta}\vert\partial_\xi^{\alpha_1+1} \lambda_{p-1}\vert\cdot\vert\partial_\xi^{\alpha_2} D_x^{\beta+1}\lambda_{p-1}\vert\right.+
\\
\nonumber
&&\left. +\sum_{\beta_1+\beta_2=\beta,\ \beta_1\neq 0}c_{\alpha_1,\alpha_2,\beta_1,\beta_2}\vert\partial_\xi^{\alpha_1+1} D_x^{\beta_1}\lambda_{p-1}\vert\cdot\vert\partial_\xi^{\alpha_2} D_x^{\beta_2+1}\lambda_{p-1}\vert\right)
\\
\nonumber
&&\leq C_{\alpha,\beta,R}\pxi_h^{-\alpha}\px^{-\beta}\pxi_h^{-1}(\ln\px+1)\px^{-1}
\\
&&\leq 2C_{\alpha,\beta,R}\cdot h^{-1}\pxi_h^{-\alpha}\px^{-\beta}, \qquad\forall \alpha,\beta\geq 0.
\eeqs
Setting $r_{p-1}(x,\xi):=r_{p-1,-1}(x,\xi) - r_{p-1,-2}(x,\xi),$
we have  also
\beqs\label{22}
|\partial_\xi^\alpha D_x^\beta r_{p-1}(x,\xi)|\leq C'_{\alpha,\beta,R}\cdot h^{-1}\pxi_h^{-\alpha}\px^{-\beta}, \qquad\forall \alpha,\beta\geq 0
\eeqs
and for some $C'_{\alpha,\beta,R}>0$; this means that for $h$ large enough, operator $I-R_{p-1}$ is invertible by Neumann series and $\sum_{n=0}^{+\infty}R_{p-1}^n$ is the inverse operator. Similar considerations hold for $e^{-\lambda_{p-1}(x,D_x)} e^{\lambda_{p-1}(x,D_x)}$; thus $e^{-\lambda_{p-1}}\sum_{n=0}^{+\infty}R_{p-1}^n$ is a left and right inverse for $e^{\lambda_{p-1}(x,D_x)}$. The lemma is then proved.
\end{proof}

\begin{Lemma}\label{lemma2bis}
Let $\lambda_{p-k}(x,\xi), k=2,\ldots, p-1$ be defined by \eqref{26}. Then for every $k$ there exists $h_k\geq 1$ such that for $h\geq h_k$ the operator $e^{\lambda_{p-k}(x,D)}$ is invertible and
\beqs
\label{18bis}
(e^{\lambda_{p-k}(x,D_x)})^{-1}=e^{-\lambda_{p-k}(x,D_x)}(I+R_{p-k})
\eeqs
where $I$ is the identity operator and $R_{p-k}$ has
principal symbol
\beqsn
r_{p-k,-k}(x,\xi)= \partial_{\xi}\lambda_{p-k}(x,\xi)D_{x}\lambda_{p-k}(x,\xi) \in SG^{-k,-\frac{p-k}{p-1}}.
\eeqsn
\end{Lemma}

\begin{proof}
The construction of the inverse is completely analogous to the one of Lemma \ref{lemma2}. Moreover, by \eqref{40}
we have that
\beqs
\nonumber
 |r_{p-k,-k}(x,\xi)|=|\partial_\xi \lambda_{p-k}(x,\xi)D_x \lambda_{p-k}(x,\xi)|&\leq&
C_{k}M_{p-k}^2\langle x\rangle^{2\frac{k-1}{p-1}-1}
\langle\xi\rangle_h^{-2k+1}\chi_{\xi}(x) 
\\
\label{questa}
&\leq& C_{k}M_{p-k}^2 \px^{-\frac{p-k}{p-1}}\pxi_h^{-k}
\\
\nonumber
&\leq& C_{k}M_{p-k}^2 h^{-1}
\eeqs
since on the support of $\chi_{\xi}(x)$ we have $\langle x\rangle\leq \langle\xi\rangle_h^{p-1}$.
The derivatives of $r_{p-k,-k}$ can be estimated similarly. Thus, for $h$ large enough we obtain \eqref{18bis}, and from \eqref{questa} we have $r_{p-k,-k} \in SG^{-k,-\frac{p-k}{p-1}}.$
\end{proof}


\section{The proof of Theorem \ref{main}}

The proof of Theorem \ref{main} needs some preparation. As announced in Section 2, we shall reduce the Cauchy problem \eqref{CP} to the problem \eqref{CP2}, where the operator $P_{\lambda}$ is defined by \eqref{Plambda} and the functions $f_{\lambda}, g_{\lambda}$ and $u_{\lambda}$ are given respectively by \eqref{fgdef} and \eqref{vdef}. We first prove the following result.

\begin{Prop} \label{propreduction}
Let $P_{\lambda}$ be defined by \eqref{Plambda}. Then we have:
$$(iP)_{\lambda}= \partial_t + ia_p(t,D_x) + \sum_{\ell=1}^{p-1}Q_{p-\ell}(t,x,D_x) + r_{0}(t,x,D_x),$$
for some operators $Q_{p-\ell}$ with symbols $Q_{p-\ell}(t,x,\xi) \in SG^{p-\ell,0},$ satisfying 
$$\langle Q_{p-\ell}v,v \rangle \geq 0 \qquad v \in \mathcal{S}(\R), \, 1\leq\ell \leq p-1$$ and for some operator
$r_0$ with symbol in $S^0.$
\end{Prop}

Proposition \ref{propreduction} will be proved in $p-1$ steps each of them corresponding to a conjugation with an operator of the form $e^{\lambda_{p-k}}, k=1,\ldots, p-1.$ 
We know by \eqref{40} and \eqref{lambdasenzader} that the operators $e^{\pm \lambda_{p-1}(x,D_x)}$ have order $(0,M_{p-1})$, while the other operators $e^{\pm \lambda_{p-k}(x,D_x)}$ have order $(0,0)$ for $k=2,\ldots, p-1.$ In the proof of the energy estimate \eqref{E} the first conjugation will play an essential role since it determines the loss of $H_{s_1,s_2}$ regularity; the others all work similarly to each other. This is the reason why we shall organize the proof of Proposition \ref{propreduction} as follows. We present in detail the first two transformations and then we argue by induction. Each of these steps corresponds to a different lemma.
Before this we give a preliminary result which states for an operator with symbol $a(x,\xi)\in SG^{m_1,0}$ the form of the composed operator
$e^{-\lambda_i(x,D_x)}a(x,D_x)e^{\lambda_i(x,D_x)}, i=1,\ldots, p-1$, with $\lambda_i$ defined as in the previous section.

%
\begin{Lemma}\label{eh} Let $a \in SG^{m_1,0}$ and let $\lambda_{i}, i=1,\ldots, p-1$ be defined by \eqref{26'}, \eqref{26}. Then the symbol of $e^{-\lambda_i(x,D_x)}a(x,D_x)e^{\lambda_i(x,D_x)}, i=1,\ldots, p-1$ is given by:
\beqs\nonumber
\left(e^{-\lambda_i}ae^{\lambda_i}\right)(x,\xi)&=&a+\ds\sum_{\alpha=1}^{m_1-1}\frac1{\alpha!}\partial_\xi^\alpha a\cdot e^{-\lambda_i}\cdot D_x^\alpha e^{\lambda_i}
\\
\label{sviluppone}
&+&\ds\sum_{\gamma=1}^{m_1-1}\frac1{\gamma!}\partial_\xi^\gamma e^{-\lambda_i}
D_x^\gamma (ae^{\lambda_i})+\ds\sum_{\gamma=1}^{m_1-2}\sum_{\alpha=1}^{m_1-1}\frac1{\alpha!\gamma!}\partial_\xi^\gamma e^{-\lambda_i}D_x^{\gamma}(\partial_\xi^\alpha a \cdot D_x^{\alpha} e^{\lambda_i})+r_0
\eeqs
with $r_0\in SG^{0,0}$. Moreover the symbol
$$r(x,\xi)=\ds\sum_{\gamma=1}^{m_1-1}\frac1{\gamma!}\partial_\xi^\gamma e^{-\lambda_i}
D_x^\gamma (ae^{\lambda_i})+\ds\sum_{\gamma=1}^{m_1-2}\sum_{\alpha=1}^{m_1-1}\frac1{\alpha!\gamma!}\partial_\xi^\gamma e^{-\lambda_i}D_x^{\gamma}(\partial_\xi^\alpha a \cdot D_x^{\alpha} e^{\lambda_i}) \in SG^{m_1-1,0}.$$
\end{Lemma}
\begin{proof} The proof easily follows by Proposition \ref{composition}.
\end{proof}

\begin{Rem} \label{compclas}
{\rm If we assume that $a \in S^{m_1}$ instead of $a \in SG^{m_1,0}$, then Lemma \ref{eh} still holds with $r \in S^{m_1-1}.$}
\end{Rem}
%
\begin{Lemma} \label{prlevp-1}
Let $h$ be as in Lemma \ref{lemma2} and consider, for $h\geq h_1,$ the operator $(iP)_1=(e^{\lambda_{p-1}(x,D_x)})^{-1}(iP)e^{\lambda_{p-1}(x,D_x)}$. There exist operators $Q_{p-1}(t,x,D_x)$, $a_{j,1}(t,x,D_x)$ and $r_{1}(t,x,D_x)$ with symbols $Q_{p-1}(t,x,\xi)\in SG^{p-1,-1}$, $a_{j,1}(t,x,\xi)\in SG^{j,-j/(p-1)}$, $1\leq j\leq p-2$ and $r_1(t,x,\xi)\in S^0$ such that
$$
(iP)_1=\partial_t+ia_p(t,D_x)+Q_{p-1}(t,x,D_x)+\ds\sum_{j=1}^{p-2}ia_{j,1}(t,x,D_x)+r_1(t,x,D_x),
$$ and $$
\Re\langle Q_{p-1}(t,x,D_x)v,v\rangle\geq 0,\ \forall v\in\mathcal S(\R).
$$
\end{Lemma}

\begin{proof}
We first notice that, by Lemma \ref{lemma2} we have
\beqs\nonumber
(iP)_1(t,x,D_x)&=&\partial_t+\ds\sum_{j=0}^p(e^{\lambda_{p-1}(x,D_x)})^{-1}ia_je^{\lambda_{p-1}(x,D_x)}=
\\
\nonumber
&=&\partial_t+\ds\sum_{j=0}^pe^{-\lambda_{p-1}(x,D_x)}ia_je^{\lambda_{p-1}(x,D_x)}+\ds\sum_{j=0}^pe^{-\lambda_{p-1}(x,D_x)}iR_{p-1}a_je^{\lambda_{p-1}(x,D_x)}
\\
\nonumber
&=&\partial_t+e^{-\lambda_{p-1}(x,D_x)}ia_pe^{\lambda_{p-1}(x,D_x)}+
\ds\sum_{j=1}^{p-1}e^{-\lambda_{p-1}(x,D_x)}i(a_j+R_{p-1}a_{j+1})e^{\lambda_{p-1}(x,D_x)}+ t_0
\\
&=&\partial_t+e^{-\lambda_{p-1}(x,D_x)}ia_pe^{\lambda_{p-1}(x,D_x)}+
\ds\sum_{j=1}^{p-1}e^{-\lambda_{p-1}(x,D_x)}i\tilde a_je^{\lambda_{p-1}(x,D_x)}+ t_0
\eeqs
for some $t_0(t,x,D_x)$ of order $(0,0)$, and with new operators $\tilde a_j=a_j+R_{p-1} a_{j+1}$ having symbol $\tilde a_j(t,x,\xi)\in SG^{j,-j/(p-1)}$. We now apply 
formula \eqref{sviluppone}, observing that in the case $i=p-1$, the term $r$ vanishes for $|\xi| \geq hR$ since it is a sum of products with at least one $\xi$-derivative of $\lambda_{p-1}$ appearing in each factor, see \eqref{800}. Then in particular, the term $r$ has order $(0,0)$ since it is compactly supported in $\xi$ and with order $0$ in $x$. Hence we obtain for the related operators:
\begin{multline}
(iP)_1=\partial_t+ ia_p+\ds\sum_{\alpha=1}^{p-1}\frac1{\alpha!}\partial_\xi^\alpha ia_p\cdot e^{-\lambda_{p-1}}\cdot D_x^\alpha e^{\lambda_{p-1}}\\ +\sum_{j=1}^{p-1}\left(i\tilde a_j+\sum_{\alpha=1}^{j-1}\frac1{\alpha!}\partial_\xi^\alpha i\tilde a_j\cdot e^{-\lambda_{p-1}}\cdot D_x^\alpha e^{\lambda_{p-1}}\right)+t_0
\\ 
=\partial_t+ ia_p+\partial_\xi a_p\partial_x\lambda_{p-1}+i\tilde a_{p-1}+\ds\sum_{\alpha=2}^{p-1}\frac1{\alpha!}\partial_\xi^\alpha ia_p\cdot e^{-\lambda_{p-1}}\cdot D_x^\alpha e^{\lambda_{p-1}}+\ds\sum_{j=1}^{p-2}i\tilde a_j
\\
+ \sum_{j=1}^{p-1}\sum_{\alpha=1}^{j-1}\frac1{\alpha!}\partial_\xi^\alpha i\tilde a_j\cdot e^{-\lambda_{p-1}}\cdot D_x^\alpha e^{\lambda_{p-1}}+\tilde{t}_0
\label{preceden} \end{multline}
for a term $\tilde{t}_0$ with symbol in $SG^{0,0}$. 
Now we consider the real part of the terms of order $p-1$ with respect to $\xi$ in \eqref{preceden}, and use \eqref{assAC2}, \eqref{assAC3},   \eqref{omegaaaaa}, \eqref{two} for $|\xi|\geq hR$ to get:
\beqsn
\Re\left(i\tilde a_{p-1}+\partial_\xi a_p\partial_x\lambda_{p-1}\right)&=&-\Im \tilde a_{p-1}+\partial_\xi a_p\partial_x\lambda_{p-1}
\\
&=&-\Im a_{p-1}+M_{p-1}\vert\partial_\xi a_p(t,\xi)\vert\px^{-1}
\\
&\geq& \left(2^{-(p-1)/2}C_pM_{p-1}-C\right)\pxi_h^{p-1}\px^{-1},
\eeqsn
where we used that $\tilde a_j=a_j$ for $|\xi|>hR,$ and we used also the inequality: $\langle a\rangle_h\leq \sqrt 2|a|$ for $a\in\R,$ $|a|\geq h$. We thus obtain
\beqsn
\Re\left(i\tilde a_{p-1}+\partial_\xi a_p\partial_x\lambda_{1}\right)\geq 0
\eeqsn
if only we choose the constant $M_{p-1}$ great enough:
\begin{equation}\label{choice!}
M_{p-1}\geq 2^{(p-1)/2 }C/C_p.
\end{equation}
We can then apply Theorem \ref{tha1} to the operator $i\tilde a_{p-1}+\partial_\xi a_p\partial_x\lambda_{p-1}$ and obtain that there exist pseudo-differential operators $Q_{p-1}(t,x,D_x)$, $\tilde{R}_{p-2}(t,x,D_x)$ and $R_0$ with symbols $Q_{p-1}(t,x,\xi)\in SG^{p-1,-1}$, $ \tilde{R}_{p-2}(t,x,\xi) \in SG^{p-2,-1}$, $R_{0}(t,x,\xi)\in S^0$ such that 
$$\Re\langle Q_{p-1}(t,x,D_x)v,v\rangle\geq 0,\ \forall v\in\mathcal S(\R)$$ and
$$i\tilde a_{p-1}+\partial_\xi a_p\partial_x\lambda_{p-1}= Q_{p-1}+ \tilde{R}_{p-2} + R_0.$$
Finally we estimate the last three terms in the right-hand side of \eqref{preceden}. We observe that for $2 \leq \alpha \leq p-1$ we have
\begin{equation}\label{stimaresti}|\partial_\xi^\alpha a_p\cdot e^{-\lambda_{p-1}}\cdot D_x^\alpha e^{\lambda_{p-1}}| \leq C_p' \pxi_h^{p-\alpha}\px^{-1-\alpha} \leq C_p' \pxi_h^{p-2}\px^{-\frac{p-2}{p-1}}, \end{equation} since $\alpha+1 \geq 3 > (p-2)/(p-1).$ Similarly we can estimate the derivatives of the above symbol observing that the worst case occurs when $\xi$-derivatives of order $\gamma$ fall on the term $e^{-\lambda_{p-1}}$. By \eqref{dermistelambdaxi} this produces a term of type $(\ln\px)^{\gamma}$ in the estimates but since $$(\ln\px)^\gamma \px^{-1-\alpha} \leq (\ln \px)^\gamma \px^{-3} \leq C \px^{-\frac{p-2}{p-1}}$$ for every $\gamma >0,$ we conclude that $\partial_\xi^\alpha a_p\cdot e^{-\lambda_{p-1}}\cdot D_x^\alpha e^{\lambda_{p-1}} \in SG^{p-2, -(p-2)/(p-1)}.$ By similar arguments we obtain for $1\leq \alpha \leq j-1$ that $\partial_\xi^\alpha \tilde{a}_j\cdot e^{-\lambda_{p-1}}\cdot D_x^\alpha e^{\lambda_{p-1}} \in SG^{j-1, -(j-1)/(p-1)}.$ Hence, we can gather the last three terms in the right-hand side of \eqref{preceden} and the remainder term $\tilde{R}_{p-2}$ obtained from the application of Theorem \ref{tha1} and conclude that
\beqsn
(iP)_1=\partial_t+ia_p(t,D_x)+Q_{p-1}(t,x,D_x)+\ds\sum_{j=1}^{p-2}ia_{j,1}(t,x,D_x)+r_{1}(t,x,D)
\eeqsn
for some $a_{j,1} \in SG^{j,-j/(p-1)}$ and $r_1 \in S^0.$
Lemma \ref{prlevp-1} is proved.\end{proof}

To perform the second transformation, we need to compute the operator 
$$(iP)_2=(e^{\lambda_{p-2}(x,D_x)})^{-1}(iP)_1e^{\lambda_{p-2}(x,D_x)}.$$ By Lemma \ref{lemma2bis}, for $h\geq h_2$ there exists an operator $R_{p-2}$ with principal symbol 
$$r_{p-2,-2}(x,\xi)=\partial_\xi\lambda_{p-2}(x,\xi)D_x\lambda_{p-2}(x,\xi)\in SG^{-2,-(p-2)/(p-1)}$$ such that:
\[(iP)_2=e^{-\lambda_{p-2}(x,D_x)}(I+R_{p-2})(iP)_1e^{\lambda_{p-2}(x,D_x)}.\]
We have the following result.

\begin{Lemma} \label{prlevp-2}
Let $h_1,h_2$ be as in Lemmas \ref{lemma2} and \ref{lemma2bis} and let $h\geq \max\{h_1,h_2\}.$ Then there exist pseudo-differential operators $Q_{p-2}(t,x,D_x)$, $a_{j,2}(t,x,D_x)$ and $r_{2}(t,x,D_x)$ with symbols $Q_{p-2}(t,x,\xi)\in SG^{p-2,0}$, $a_{j,2}(t,x,\xi)\in SG^{j,-j/(p-1)}$ for $1\leq j\leq p-3$, $r_2(t,x,\xi)\in S^0$ such that:
$$(iP)_2=\partial_t+ia_p(t,D_x)+Q_{p-1}(t,x,D_x)+ Q_{p-2}(t,x,D_x)+\ds\sum_{j=1}^{p-3}ia_{j,2}(t,x,D_x)+r_2(t,x,D_x),$$
and 
$$\Re\langle Q_{p-2}(t,x,D_x)v,v\rangle\geq 0,\ \forall v\in\mathcal S(\R). $$
\end{Lemma}

\begin{proof}
By Lemma \ref{prlevp-1} we have (omitting the notation $(t,x,D_x)$ ):
\beqs \label{iP2}
(iP)_2&=&\partial_t+e^{-\lambda_{p-2}}\left(ia_p+Q_{p-1}+\ds\sum_{j=1}^{p-2}ia_{j,1}+r_1 \right)e^{\lambda_{p-2}}\\
\nonumber
&&+e^{-\lambda_{p-2}}\left( iR_{p-2}a_p+R_{p-2}Q_{p-1}+\ds\sum_{j=1}^{p-2}iR_{p-2}a_{j,1}+R_{p-2}r_1 \right)e^{\lambda_{p-2}}.
\eeqs
where $a_{j,1}(t,x,D)$ have symbols in $SG^{j,-j/(p-1)},$ and $r_1$ has symbol in $S^0.$
Now we observe that $iR_{p-2}a_p \in SG^{p-2, -(p-2)/(p-1)},$ $R_{p-2}Q_{p-1} \in SG^{p-3, -(p-3)/(p-1)}$ and $R_{p-2}ia_{j,1} \in SG^{j-2,-(j-2)/(p-1)}$ for every $j=1,\ldots, p-2.$ Then we can write \eqref{iP2} as follows:
$$ (iP)_2=\partial_t+e^{-\lambda_{p-2}(x,D_x)}\left(ia_p+Q_{p-1}+\ds\sum_{j=1}^{p-2}i\tilde a_{j,1}\right)e^{\lambda_{p-2}(x,D_x)}+t_0$$
for some $t_0$ with symbol in $S^0$ and with new symbols $\tilde a_{j,1}\in SG^{j,-j/(p-1)}$. Here we have considered the composition $e^{\lambda_{p-2}}(R_{p-2}r_1)e^{\lambda_{p-2}}$ using Remark \ref{compclas}.
Here it is important to underline that the operator $i\tilde a_{p-2,1}$ has symbol:
\begin{equation} \label{ap-2,1}
i\tilde a_{p-2,1}(t,x,\xi)=i a_{p-2,1}(t,x,\xi)+ia_p(t,\xi)\partial_\xi\lambda_{p-2}(t,x,\xi)D_x\lambda_{p-2}(t,x,\xi),
\end{equation}
where the first term in the right-hand side depends only on the constant $M_{p-1}$ and the second only on $M_{p-2}$ which will be chosen later on in this second transformation. But now, by formula \eqref{sviluppone}, we get:
\beqsn \label{formulone}
(iP)_2&=&\partial_t+ ia_p+Q_{p-1}+i\tilde{a}_{p-2,1}+\partial_\xi a_p\partial_x\lambda_{p-2} +\ds\sum_{j=1}^{p-3}i\tilde a_{j,1}
\\
&&+\ds\sum_{\alpha=2}^{p-1}\frac1{\alpha!}\partial_\xi^\alpha ia_p\cdot e^{-\lambda_{p-2}}\cdot D_x^\alpha e^{\lambda_{p-2}}-ia_p\partial_\xi\lambda_{p-2}D_x\lambda_{p-2}+ \sum_{\gamma=2}^{p-1} \frac{1}{\gamma!}\partial_{\xi}^{\gamma}e^{-\lambda_{p-2}} \cdot (ia_pD_x^{\gamma} e^{\lambda_{p-2}}) \\ && +\ds\sum_{\gamma=1}^{p-2}\sum_{\alpha=1}^{p-1}\frac{1}{\alpha!\gamma!}\partial_\xi^\gamma e^{-\lambda_{p-2}}D_x^{\gamma}(\partial_\xi^\alpha a_p D_x^{\alpha} e^{\lambda_{p-2}})
\\
&&+\ds\sum_{\alpha=1}^{p-2}\frac1{\alpha!}\partial_\xi^\alpha Q_{p-1}\cdot e^{-\lambda_{p-2}} \cdot D_x^\alpha e^{\lambda_{p-2}}
+\ds\sum_{j=1}^{p-2}\sum_{\alpha=1}^{j-1}\frac1{\alpha!}\partial_\xi^\alpha i\tilde{a}_{j,1} \cdot e^{-\lambda_{p-2}}\cdot D_x^\alpha e^{\lambda_{p-2}}
\\
&&
+\ds\sum_{\gamma=1}^{p-2}\frac1{\gamma!}\partial_\xi^\gamma e^{-\lambda_{p-2}}D_x^\gamma (Q_{p-1}e^{\lambda_{p-2}})+\sum_{\alpha=1}^{p-2}\ds\sum_{\gamma=1}^{p-3}\frac1{\alpha!\gamma!}\partial_\xi^\gamma e^{-\lambda_{p-2}}D_x^\gamma(\partial_\xi^\alpha Q_{p-1} D_x^{\alpha} e^{\lambda_{p-2}})
\\
&&+\ds\sum_{j=1}^{p-2}\sum_{\gamma=1}^{j-1}\frac1{\gamma!}\partial_\xi^\gamma e^{-\lambda_{p-2}}D_x^\gamma (i\tilde{a}_{j,1}e^{\lambda_{p-2}})+\ds\sum_{j=1}^{p-2}\sum_{\alpha=1}^{j-1}\sum_{\gamma=1}^{j-2}\frac1{\alpha!\gamma!}\partial_\xi^\gamma e^{-\lambda_{p-2}}D_x^\gamma(\partial_\xi^\alpha ( i\tilde{a}_{j,1}) D_x^{\alpha} e^{\lambda_{p-2}}) \\ &&
+t_0.
\eeqsn
By estimating as before the orders of the terms appearing in the above formula, using (\ref{40}) and by \eqref{ap-2,1} we obtain that
\beqs
(iP)_2=\partial_t+ ia_p+Q_{p-1}+\partial_\xi a_p\partial_x\lambda_{p-2}+ia_{p-2,1}+\ds\sum_{j=1}^{p-3}i\tilde{\tilde{a}}_{j,1}+\tilde{t}_0,
\eeqs
for some operators $\tilde{\tilde{a}}_{j,1}(t,x,D)$ with symbols $\tilde{\tilde{a}}_{j,1}(t,x,\xi)\in SG^{j,-j/(p-1)}$, and with a term $\tilde{t}_0$ with symbol in $S^0$. We want now to apply Theorem \ref{tha1} to the operator $ia_{p-2,1}+ \partial_\xi a_p\partial_x\lambda_{p-2}$, namely to the term of order $p-2$ with respect to $\xi$. 
We observe that for $|\xi| \geq hR$ we have:
\begin{eqnarray*}
\textrm{Re}(ia_{p-2,1}+ \partial_\xi a_p\partial_x\lambda_{p-2})
 &=& -\textrm{Im}a_{p-2,1} +  \partial_\xi a_p\partial_x\lambda_{p-2}
\\ &=& -\textrm{Im}a_{p-2,1} +M_{p-2} \partial_{\xi}a_p \omega \left( \frac{\xi}{h} \right) \px^{-\frac{p-2}{p-1}} \psi \left( \frac{\px}{\pxi_h^{p-1}}\right) \pxi_h^{-1} \\ 
&=&  -\textrm{Im}a_{p-2,1} +M_{p-2}|\partial_{\xi}a_p(t,\xi)| \pxi_h^{-1}\px^{-\frac{p-2}{p-1}}\psi \left( \frac{\px}{\pxi_h^{p-1}}\right) \\ &\geq& -C \pxi_h^{p-2}\px^{-\frac{p-2}{p-1}}+M_{p-2}
C_p |\xi|^{p-1} \pxi_h^{-1}\px^{-\frac{p-2}{p-1}}\psi \left( \frac{\px}{\pxi_h^{p-1}}\right) 
\\ & \geq &
-C\pxi_h^{p-2}\px^{-\frac{p-2}{p-1}}+M_{p-2}
C_p 2^{-(p-1)/2}\pxi_h^{p-2} \px^{-\frac{p-2}{p-1}}\psi \left( \frac{\px}{\pxi_h^{p-1}}\right) 
\\ &=& (-C+M_{p-2}C_p 2^{-(p-1)/2})
\pxi_h^{p-2} \px^{-\frac{p-2}{p-1}}\psi \left( \frac{\px}{\pxi_h^{p-1}}\right)
\\ &&- C \pxi_h^{p-2} \px^{-\frac{p-2}{p-1}} \left(1-\psi \big( \frac{\px}{\pxi_h^{p-1}}\big)\right) \\
&\geq& (-C+M_{p-2}C_p2^{-(p-1)/2})
\pxi_h^{p-2} \px^{-\frac{p-2}{p-1}}\psi \left( \frac{\px}{\pxi_h^{p-1}}\right)-C',
\end{eqnarray*}
where $C=C(M_{p-1})$ is a constant depending on the already chosen $M_{p-1}>0$ and, in the last inequality we used the fact that on the support of $1-\psi\big( \frac{\px}{\pxi_h^{p-1}}\big)$ we have 
$\pxi_h^{p-2} \px^{-\frac{p-2}{p-1}} \leq C'.$
Choosing the constant $M_{p-2} \geq C(M_{p-1})2^{(p-1)/2}/C_p$, we obtain that $$\textrm{Re} (ia_{p-2,1}+ \partial_\xi a_p\partial_x\lambda_{p-2}) \geq -C' \quad \textrm{for} \quad |\xi| \geq hR.$$
We can then apply Theorem \ref{tha1} to the symbol $ia_{p-2,1}+ \partial_\xi a_p\partial_x\lambda_{p-2} +C' \in SG^{p-2,0}$. Then there exist operators $Q_{p-2}(t,x,D), \tilde{R}_{p-3}, R_0$ such that 
\beqsn
&&ia_{p-2,1}+ \partial_\xi a_p\partial_x\lambda_{p-2} +C' = Q_{p-2} + \tilde{R}_{p-3} + R_0, \\
&&\Re\langle   Q_{p-2}v,v \rangle  \geq 0 \quad \forall v \in \mathcal{S}(\R).
\eeqsn
It is now crucial to observe that $Q_{p-2}(t,x,\xi) \in SG^{p-3,0},$ $R_0 \in S^{0}$ whereas by \eqref{restosharpgarding} we have $\tilde{R}_{p-3} \in SG^{p-3, -(p-3)/(p-1)}$ since the constant $C'$ does not appear in the expression of the symbol. As a matter of fact we have
$$\tilde{r}_{p-3}(x,\xi)=\psi_1(\xi)D_x (ia_{p-2,1}+ \partial_\xi a_p\partial_x\lambda_{p-2})+
\sum_{|\alpha+\beta|\geq2}\psi_{\alpha,\beta}(\xi) \partial_\xi^\alpha D_x^\beta (ia_{p-2,1}+ \partial_\xi a_p\partial_x\lambda_{p-2}),$$
which is in $SG^{p-3, -(p-3)/(p-1)}.$ The proposition is proved. 
\end{proof}

 We now prove by induction on $n\geq 2$ the following:
 \begin{Lemma} \label{prlevp-n}
Given $n\in\N$, $n\geq 2$, we can find a constant $h_n\geq 1$ such that for $h\geq h_n$ there exist pseudo-differential operators $Q_{p-\ell}(t,x,D_x)$, $1\leq \ell\leq n$, $a_{j,n}(t,x,D_x)$, $1\leq j\leq p-n-1$, and $r_{n}(t,x,D_x)$ with symbols $Q_{p-\ell}(t,x,\xi)\in SG^{p-\ell,0}$, $a_{j,n}(t,x,\xi)\in SG^{j,-j/(p-1)}$, $r_n(t,x,\xi)\in S^0$ such that the operator
\[(iP)_n=:(e^{\lambda_{p-n}(x,D_x)})^{-1}\ldots(e^{\lambda_{p-2}(x,D_x)})^{-1}(e^{\lambda_{p-1}(x,D_x)})^{-1}(iP)e^{\lambda_{p-1}(x,D_x)}e^{\lambda_{p-2}(x,D_x)}\ldots e^{\lambda_{p-n}(x,D_x)}\]
can be written in the form:
$$(iP)_n=\partial_t+ia_p(t,D_x)+\ds\sum_{\ell=1}^nQ_{p-\ell}(t,x,D_x)+\ds\sum_{j=1}^{p-n-1}ia_{j,n}(t,x,D_x)+r_n(t,x,D_x),$$
and 
\begin{equation}\label{shQp-l}
\Re\langle Q_{p-\ell}(t,x,D_x)v,v\rangle\geq 0,\ \forall v\in\mathcal S(\R),\; 1\leq\ell\leq n. 
\end{equation}
\end{Lemma}
\begin{proof}
For $n=2$ this is exactly the thesis of Lemma \ref{prlevp-2}. Let us suppose now that for $h\geq h_{n-1}$
it holds
\begin{equation}\label{prlevp-n+1}
(iP)_{n-1}=\partial_t+ia_p(t,D_x)+\ds\sum_{\ell=1}^{n-1}Q_{p-\ell}(t,x,D_x)+\ds\sum_{j=1}^{p-n}ia_{j,{n-1}}(t,x,D_x)+r_{n-1}(t,x,D_x),
\end{equation}
for some pseudo-differential operators $Q_{p-\ell}(t,x,D_x)$, $1\leq \ell\leq n-1$, $a_{j,n-1}(t,x,D_x)$, $1\leq j\leq p-n$, and $r_{n-1}(t,x,D_x)$ with symbols $Q_{p-\ell}(t,x,\xi)\in SG^{p-\ell,0}$, $a_{j,n-1}(t,x,\xi)\in SG^{j,-j/(p-1)}$, $r_{n-1}(t,x,\xi)\in S^0$ such that 
$$\Re\langle Q_{p-\ell}(t,x,D_x)v,v\rangle\geq 0,\ \forall v\in\mathcal S(\R),\; 1\leq\ell\leq n-1, $$
and consider the operator
$$(iP)_n=(e^{\lambda_{p-n}(x,D_x)})^{-1}(iP)_{n-1}e^{\lambda_{p-n}(x,D_x)}.$$ 
Lemma \ref{lemma2bis} gives, for $h$ great enough, say $h\geq \tilde h_n$, that
\[(iP)_n=e^{-\lambda_{p-n}(x,D_x)}(I+R_{p-n})(iP)_{n-1}e^{\lambda_{p-n}(x,D_x)},\]
with $R_{p-n}$ a pseudo-differential operator with principal symbol 
$$r_{p-n,-n}(x,\xi)=\partial_\xi\lambda_{p-n}(x,\xi)D_x\lambda_{p-n}(x,\xi)\in SG^{-n,-(p-n)/(p-1)}.$$ 
Thus, for $h\geq h_n=\max\{\tilde h_n, h_{n-1}\}$ and by \eqref{prlevp-n+1} we have (omitting $(t,x,D_x)$ in the notation):
\begin{eqnarray*} \label{iPn}
(iP)_n&=&\partial_t+e^{-\lambda_{p-n}}\left(ia_p+\ds\sum_{\ell=1}^{n-1}Q_{p-\ell}+\ds\sum_{j=1}^{p-n}ia_{j,n-1}+r_{n-1}\right)e^{\lambda_{p-n}}\\
&&+e^{-\lambda_{p-n}}\left( iR_{p-n}a_p+\ds\sum_{\ell=1}^{n-1}R_{p-n}Q_{p-\ell}+\ds\sum_{j=1}^{p-n}iR_{p-n}a_{j,n-1}+R_{p-n}r_{n-1}\right)e^{\lambda_{p-n}}.
\end{eqnarray*}
Now we notice that $R_{p-n}Q_{p-\ell} \in SG^{p-n-\ell, -(p-n)/(p-1)} \subset SG^{p-n, -(p-n)/(p-1)}$ since $1\leq \ell\leq n-1$, $iR_{p-n}a_p \in SG^{p-n, -(p-n)/(p-1)},$ and $R_{p-n}ia_{j,n-1} \in SG^{j-n,-(p-n)/(p-1)}\subset SG^{j-n,-(j-n)/(p-1)}$, $j=1,\ldots, p-n$;  we can so write $(iP)_n$ as follows:
$$ (iP)_n=\partial_t+e^{-\lambda_{p-n}(x,D_x)}\left(ia_p+\ds\sum_{\ell=1}^{n-1}Q_{p-\ell}+\ds\sum_{j=1}^{p-n}i\tilde a_{j,n-1}\right)e^{\lambda_{p-n}(x,D_x)}+t_0$$
for some $t_0$ with symbol in $S^0$, and with new symbols $\tilde a_{j,n-1}\in SG^{j,-j/(p-1)}$. Again, we need to notice that at level $p-n$ we get:
$$
i\tilde a_{p-n,n-1}(t,x,\xi)=i a_{p-n,n-1}(t,x,\xi)+ia_p(t,\xi)\partial_\xi\lambda_{p-n}(t,x,\xi)D_x\lambda_{p-n}(t,x,\xi)
$$
where the first term in the right-hand side depends on the constants $M_{p-1}, \ldots, M_{p-n+1}$ chosen before and the second one depends only on the constant $M_{p-n}$ which will be chosen later on; 
this does not cause any trouble in the choice of the constant $M_{p-n}$ because, again by \eqref{sviluppone}, we get:
\beqs 
\nonumber
(iP)_n&=&\partial_t+ ia_p+\ds\sum_{\ell=1}^{n-1}Q_{p-\ell}+i\tilde{a}_{p-n,n-1}+\partial_\xi a_p\partial_x\lambda_{p-n} +\ds\sum_{j=1}^{p-n-1}i\tilde a_{j,n-1}
\\
\nonumber
&&+\ds\sum_{\alpha=2}^{p-n}\frac1{\alpha!}\partial_\xi^\alpha ia_p\cdot e^{-\lambda_{p-n}}\cdot D_x^\alpha e^{\lambda_{p-n}}-ia_p\partial_\xi\lambda_{p-n}D_x\lambda_{p-n}
\\
\nonumber
 &&+ \sum_{\gamma=2}^{p-n} \frac{1}{\gamma!}\partial_{\xi}^{\gamma}e^{-\lambda_{p-n}} \cdot (ia_pD_x^{\gamma} e^{\lambda_{p-n}})  +\ds\sum_{\gamma=1}^{p-n}\sum_{\alpha=1}^{p-n+1}c_{\alpha,\gamma}\partial_\xi^\gamma e^{-\lambda_{p-n}}\partial_\xi^\alpha a_p D_x^{\alpha+\gamma} e^{\lambda_{p-n}}
\\
\nonumber
&&+\ds\sum_{\ell=1}^{n-1}\ds\sum_{\alpha=1}^{p-\ell-n}\frac1{\alpha!}\partial_\xi^\alpha Q_{p-\ell}\cdot e^{-\lambda_{p-n}} \cdot D_x^\alpha e^{\lambda_{p-n}}
+\ds\sum_{j=1}^{p-n}\sum_{\alpha=1}^{j-1}\frac1{\alpha!}\partial_\xi^\alpha i\tilde{a}_{j,n-1} \cdot e^{-\lambda_{p-n}}\cdot D_x^\alpha e^{\lambda_{p-n}}
\\
\nonumber
&&
+\ds\sum_{\ell=1}^{n-1}\ds\sum_{\gamma=1}^{p-\ell}\frac1{\gamma!}\partial_\xi^\gamma e^{-\lambda_{p-n}}D_x^\gamma (Q_{p-\ell}e^{\lambda_{p-n}})+\ds\sum_{\ell=1}^{n-1}\sum_{\alpha=1}^{p-n}\ds\sum_{\gamma=1}^{p-\ell-1}\frac1{\alpha!\gamma!}\partial_\xi^\gamma e^{-\lambda_{p-n}}D_x^\gamma(\partial_\xi^\alpha Q_{p-\ell} D_x^{\alpha} e^{\lambda_{p-n}})
\\
\nonumber
&&+\ds\sum_{j=1}^{p-n}\sum_{\gamma=1}^{j-1}\frac1{\gamma!}\partial_\xi^\gamma e^{-\lambda_{p-n}}D_x^\gamma (i\tilde{a}_{j,n-1}e^{\lambda_{p-n}})+\ds\sum_{j=1}^{p-n}\sum_{\alpha=1}^{j-1} \ds \sum_{\gamma=1}^{j-2}\frac1{\gamma!}\partial_\xi^\gamma e^{-\lambda_{p-n}}D_x^\gamma(\partial_\xi^\alpha (i\tilde{a}_{j,n-1}) D_x^{\alpha} e^{\lambda_{p-n}})
\\
\nonumber
&&+\tilde{t}_0
\\
\label{ciao!}
&=&\partial_t+ ia_p+\ds\sum_{\ell=1}^{n-1}Q_{p-\ell}+\left(\partial_\xi a_p\partial_x\lambda_{p-n}+ia_{p-n,n-1}\right)+\ds\sum_{j=1}^{p-n-1}i\tilde{\tilde{a}}_{j,n-1}+\tilde{t}_0,
\eeqs
for some operators $\tilde{\tilde{a}}_{j,n-1}(t,x,D_x)$ with symbols $\tilde{\tilde{a}}_{j,n-1}(t,x,\xi)\in SG^{j,-j/(p-1)}$, $1\leq j\leq p-n-1$,
depending on $M_{p-1}, \ldots, M_{p-n+1}$ and where $\tilde{t}_0$ is a term containing operators with symbol in $S^0$.

As done in the second transformation, we now look at the real part of the terms of order $p-n$ with respect to $\xi$ in (\ref{ciao!}); for $|\xi| \geq hR$ we have by (\ref{assAC3}),(\ref{26}) and for a positive constant $C=C(M_{p-1},\ldots,M_{p-n+1})$: 
\begin{eqnarray*}
\textrm{Re}(ia_{p-n,n-1}+ \partial_\xi a_p\partial_x\lambda_{p-n})
 &=& -\textrm{Im}a_{p-n,n-1} +M_{p-n} \partial_{\xi}a_p \omega \left( \frac{\xi}{h} \right) \px^{-\frac{p-n}{p-1}} \psi \left( \frac{\px}{\pxi_h^{p-1}}\right) \pxi_h^{-n+1} 
 \\ 
&\geq& -C \pxi_h^{p-n}\px^{-\frac{p-n}{p-1}}+M_{p-n}
C_p |\xi|^{p-1} \pxi_h^{-n+1}\px^{-\frac{p-n}{p-1}}\psi \left( \frac{\px}{\pxi_h^{p-1}}\right) 
\\
&\geq& (-C+M_{p-n}C_p2^{-(p-1)/2})
\pxi_h^{p-n} \px^{-\frac{p-n}{p-1}}\psi \left( \frac{\px}{\pxi_h^{p-1}}\right)-C'
\\
&\geq & -C'
\end{eqnarray*}
if we choose $M_{p-n} \geq C(M_{p-1},\ldots,M_{p-n+1})2^{(p-1)/2}/C_p$. An application of Theorem \ref{tha1} to the symbol $ia_{p-n,n-1}+ \partial_\xi a_p\partial_x\lambda_{p-n} +C' \in SG^{p-n,0}$ gives then the existence of operators $Q_{p-n}(t,x,D), \tilde{R}_{p-n-1}(t,x,D), R_0(t,x,D)$ with symbols respectively $Q_{p-n}(t,x,\xi) \in SG^{p-n,0},$ $\tilde{R}_{p-n-1} \in$ $ SG^{p-n-1, -(p-n)/(p-1)}$ , $R_0 \in S^{0}$ such that 
\beqsn 
&&ia_{p-n,n-1}+ \partial_\xi a_p\partial_x\lambda_{p-n}+C'=Q_{p-n}+\tilde R_{p-n-1}+ R_0,
\\
&&\Re\langle   Q_{p-n}v,v \rangle  \geq 0 \quad \forall v \in \mathcal{S}(\R).
\eeqsn
Then, from \eqref{ciao!} we finally obtain
\beqsn
(iP)_n&=&\partial_t+ia_p+\ds\sum_{\ell=1}^{n-1}Q_{p-\ell}+ Q_{p-n} + \tilde{R}_{p-n-1} +\ds\sum_{j=1}^{p-n-1}i\tilde{\tilde{a}}_{j,n-1}+t_0+ R_0
\\
&=& \partial_t+ia_p+\ds\sum_{\ell=1}^{n}Q_{p-\ell} +\ds\sum_{j=1}^{p-n-1}ia_{j,n}+r_n
\eeqsn
for some operators $a_{j,n}$, $r_n$ with symbols $a_{j,n}\in SG^{j,-j/(p-1)}$, $r_n\in S^{0}$ and where $Q_{p-\ell}$ satisfy \eqref{shQp-l} for $1\leq \ell\leq n$.
The Lemma is proved. 
\end{proof}

Proposition \ref{propreduction} follows directly from Lemmas  \ref{prlevp-1}, \ref{prlevp-2} and \ref{prlevp-n}.

\begin{proof}[Proof of Theorem \ref{main}] By the change of variable \eqref{vdef} the Cauchy problem \eqref{CP} in the unknown $u(t,x) $ is reduced to the Cauchy problem \eqref{CP2} for the unknown $u_{\lambda}(t,x)$, where $P_{\lambda}$ is defined by \eqref{Plambda} and $f_\lambda$ and $g_\lambda$ are defined by \eqref{fgdef}. We now apply Proposition \ref{propreduction} to derive an energy estimate for the problem \eqref{CP2}: for every $v \in C^1([0,T], \mathcal{S}(\R))$ we have
\beqsn
\frac{d}{dt}\|v\|_0^2&=&2\Re\langle\partial_tv,v\rangle
\\
&=&2\Re\langle (iP_{\lambda})v,v\rangle-2\ds\sum_{\ell=1}^{p-1}\Re\langle Q_{p-\ell} v,v\rangle-2\Re\langle r_{0}v,v\rangle
\\
&\leq& C(\|P_{\lambda}v\|_{L^2}^2+\|v\|_{L^2}^2).
\eeqsn
By standard arguments in the energy method we deduce that, for all $s=(s_1,s_2)\in\R^2$ and every $v\in C^1([0,T],\mathcal{S}(\R))$ the following estimate holds
\beqs
\label{E3}
\|v(t,\cdot)\|_{s_1,s_2}^2\leq c'\left(\|v(0,\cdot)\|_{s_1,s_2}^2+
\int_0^t\|P_{\lambda}v(\tau,\cdot)\|_{s_1,s_2}^2d\tau\right)
\qquad\forall t\in[0,T],
\eeqs
for some $c'>0$.
The energy estimate \eqref{E3} can be extended by a density argument to a function $v \in C^1([0,T], H_{s_1,s_{2}}(\R))$ for every $s_1,s_2 \in \R$. This implies that if $f_\lambda \in C([0,T], H_{s_1,s_2}(\R)), g_\lambda \in H_{s_1,s_2}(\R),$ then the Cauchy problem \eqref{CP2} has a unique solution $u_{\lambda}(t,x) \in C^1([0,T], H_{s_1,s_2}(\R))$ $\cap C([0,T], H_{s_1+p,s_2}(\R))$ satisfying \eqref{E3}. 
By the relation \eqref{vdef} between $u$ and $u_{\lambda}$ and by Lemma \ref{Lemma2.5}, we obtain existence and uniqueness of a solution $u$ of \eqref{CP}. Moreover, from \eqref{fgdef}, \eqref{vdef} and \eqref{E3} we get
\beqs\label{estloss}
\|u\|_{s_1,s_2-2M_{p-1}}^2
\leq&&c_1\|u_{\lambda}\|_{s_1,s_2-M_{p-1}}^2\leq c_2
\left(\|g_{\lambda}\|_{s_1,s_2-M_{p-1}}^2+
\int_0^t\|f_{\lambda}\|_{s_1,s_2-M_{p-1}}^2d\tau\right)
\\
\nonumber
\leq&&c_3
\left(\|g\|_{s_1,s_2}^2+\int_0^t\|f\|_{s_1,s_2}^2d\tau\right)
\eeqs
for some $c_1,c_1,c_3>0$.
This gives the energy estimate (\ref{E}), and proves well posedness in $\mathcal S$, $\mathcal S'$ of the Cauchy problem (\ref{CP}). \end{proof}

\begin{Rem} \label{finalrem}
{\rm As outlined in the Introduction, Theorem \ref{main} can be proved replacing the assumption \eqref{assAC3} by the weaker condition \eqref{refinement} and repeating readily the argument of the proof above. For the sake of brevity we leave the details to the reader. We limit ourselves to observe that the assumption \eqref{refinement} is sufficient for the application of Theorem \ref{tha1} in the proofs of Lemmas \ref{prlevp-1}, \ref{prlevp-2}, \ref{prlevp-n} where only the imaginary parts of the symbols $a_j(t,x,\xi)$ are involved and that every step can be performed identically with the only difference that the symbols $a_{j,\ell}, \ell=1,\ldots, p-1$ appearing in the above Lemmas are now such that $\Re a_{j,\ell}\in SG^{j,0}$ and $\Im a_{j,\ell}\in SG^{j,-j/(p-1)}.$ Nevertheless this is sufficient to obtain the assertion of Proposition \ref{propreduction}. }
\end{Rem}

\end{section}

\end{document}